\documentclass[11pt,a4paper,thmsb]{article}%
\usepackage{graphicx}
\usepackage{amsmath}
\usepackage{amsfonts}
\usepackage{amssymb}%
\providecommand{\U}[1]{\protect\rule{.1in}{.1in}}
\providecommand{\U}[1]{\protect\rule{.1in}{.1in}}
\newtheorem{theorem}{Theorem}

\newtheorem{corollary}[theorem]{Corollary}

\newtheorem{definition}[theorem]{Definition}

\newtheorem{lemma}[theorem]{Lemma}

\newtheorem{proposition}[theorem]{Proposition}
\newtheorem{remark}[theorem]{Remark}

\newenvironment{proof}[1][Proof]{\textbf{#1.} }{\rule{0.5em}{0.5em}}
\begin{document}

\title{On isolated strata of p-gonal Riemann surfaces in the branch locus of
moduli spaces}
\author{Gabriel Bartolini\\{\small Matematiska institutionen, }\\{\small Link\"{o}pings universitet, }\\{\small 581 83 Link\"{o}ping, Sweden.}
\and Antonio F. Costa\thanks{Partially supported by MTM2011-23092}\\{\small Dept. Matem\'{a}ticas Fundamentales, }\\{\small Facultad de Ciencias, UNED, }\\{\small 28040 Madrid, Spain.}
\and Milagros Izquierdo\thanks{Partially supported by the Royal Swedish Academy of Sciences. Work done during a visit to the Institut Mittag-Leffler (Djursholm, Sweden)}\\{\small Matematiska institutionen, }\\{\small Link\"{o}pings universitet, }\\{\small 581 83 Link\"{o}ping, Sweden.}}
\date{}
\maketitle

\begin{abstract}
The moduli space $\mathcal{M}_{g}$ of compact Riemann surfaces of genus $g$
has orbifold structure, and the set of singular points of such orbifold is the
\textit{branch locus} $\mathcal{B}_{g}$. Given a prime number $p \ge 7$, $\mathcal{B}_{g}$ contains isolated strata
consisting of $p$-gonal Riemann surfaces for genera $g \ge {3(p-1)\over 2}$,that are multiple of ${p-1 \over 2}$. 
This is a generalization of the results obtained in \cite{[BCI1]} for pentagonal Riemann surfaces, and the results of \cite{[K]} and \cite{[CI3]} for zero- and one-dimensional isolated strata in the branch locus.

\end{abstract}

\section{Introduction}

In this article we study the topology of moduli spaces of Riemann surfaces.
The moduli space $\mathcal{M}_{g}$ of compact Riemann surfaces of genus $g$, 
being the quotient of the Teichm\"{u}ller space by the discontinuous action of
the mapping class group, has the structure of a complex orbifold, whose set of
singular points is called the \textit{branch locus} $\mathcal{B}_{g}$. The
branch locus $\mathcal{B}_{g}$, $g\geq3$ consists of the Riemann surfaces with
symmetry, i. e. Riemann surfaces with non-trivial automorphism group. 
$\mathcal{B}_{g}$ admits an (equisymmetric) stratification where each stratum is given by the symmetry 
of the surfaces in it, i.e. the conjugacy class in the mapping class group of the automorphism group of the surfaces of the stratum (\cite{[B]}).

 Our goal
is to study the topology of $\mathcal{B}_{g}$ through its connectedness, using this equisymmetric stratification.
The connectedness of moduli spaces of hyperelliptic, $p-$gonal and real
Riemann surfaces has been widely studied, for instance by \cite{[BSS]}, \cite{[K]}
\cite{[CI1]}, \cite{[CI2]}, \cite{[CI3]}, \cite{[BCI2]}, \cite{[G]}, \cite{[S]}, \cite{[BCIP]} and \cite{[BEMS]}.

Recently Bartolini, Costa and Izquierdo have shown that $\mathcal{B}_{g}$ is connected 
only for genera 3, 4, 7, 13, 17, 19 and 59 (see \cite{[BCI1]} and \cite{[BCI2]}). 
The authors found isolated strata in $\mathcal{B}_g$ 
($g \neq 3, 4, 7, 13, 17, 19, 59$) given by actions of order five and seven. 
In \cite{[BI]} it is shown that the strata induced by actions of order two and three 
belongs to the same connected component of $\mathcal{B}_{g}$. 

A cyclic $p$-gonal Riemann surface $X$ is a surface that admits a regular covering of 
degree $p$ on the Riemann sphere. A 2-gonal Riemann surface is called an hyperelliptic Riemann surface.

The main result in this article is that $\mathcal{B}_{g}$ 
contains isolated strata consisting of $p$-gonal Riemann surfaces ($p\ge 7$) of dimension $d\ge 2$ for genus 
$g=(d+1)({p-1\over 2})$,  
according to Riemann-Hurwitz's formula. 

Given two Riemann surfaces $X_1$ and $X_2$, there is a path of quasiconformal deformations taking $X_1$ to $X_2$ since 
$\mathcal{M}_{g}$ is connected. The result obtained in this article says that if $X_1$ belongs to one of the isolated strata and $X_2$ has another type of symmetry, then the path of quasiconformal deformations must contain surfaces without symmetry.
 
The main result is a generalization of the results obtained 
in \cite{[BCI1]} for isolated strata of cyclic pentagonal Riemann surfaces, and of the results in \cite{[K]}, \cite{[CI3]} for isolated strata of dimension zero and one. 
As a consequence we give an infinite family of genera for which $\mathcal{B}_{g}$ has an increasing number of 
isolated strata.


\section{Riemann surfaces and Fuchsian groups}

Let $X$ be a Riemann surface and assume that $Aut(X)\neq\{1\}$. Hence
$X/Aut(X)$ is an orbifold and there is a Fuchsian group $\Gamma\leq
Aut(\mathcal{D})$, such that $\pi_{1}(X)\vartriangleleft\Gamma$:%
\[
\mathcal{D\rightarrow}X=\mathcal{D}/\pi_{1}(X)\rightarrow X/Aut(X)=\mathcal{D}%
/\Gamma
\]
where $\mathcal{D}=\{z\in\mathbb{C}:\left\Vert z\right\Vert <1\}$.

If the Fuchsian group $\Gamma$ is isomorphic to an abstract group with
canonical presentation%

\begin{equation}
\left\langle a_{1},b_{1},\dots,a_{g},b_{g},x_{1}\dots x_{k}|x_{1}^{m_{1}%
}=\dots=x_{k}^{m_{k}}=\prod_{i=1}^{k}x_{i}\prod_{i=1}^{g}[a_{i},b_{i}%
]=1\right\rangle ,\label{presentation}%
\end{equation}
we say that $\Gamma$ has \emph{signature}
\begin{equation}
s(\Gamma)=(g;m_{1},\dots,m_{k}).\label{sign}%
\end{equation}
\noindent The generators in presentation (\ref{presentation}) will be called
\textit{canonical generators}.

Let $X$ be a Riemann surface uniformized by a surface Fuchsian group
$\Gamma_{g}$, i.e. a group with signature $(g;-)$. A finite group $G$ is a
group of automorphisms of $X$, i.e. there is a holomorphic action $a$ of $G$
on $X$, if and only if there is a Fuchsian group $\Delta$ and an epimorphism
$\theta_{a}:\Delta\rightarrow G$ such that $\ker\theta_{a}=\Gamma_{g}$. The
epimorphism $\theta_{a}$ is the monodromy of the covering $f_{a}:X\rightarrow
X/G=\mathcal{D}/\Delta$.

The relationship between the signatures of a Fuchsian group and subgroups is
given in the following theorem of Singerman:

\begin{theorem}
\label{subgroupthm} (\cite{[Si1]}) Let $\Gamma$ be a Fuchsian group with
signature (\ref{sign}) and canonical presentation (\ref{presentation}). Then
$\Gamma$ contains a subgroup $\Gamma^{\prime}$ of index $N$ with signature
\[
s({\Gamma}^{\prime})=(h;m_{11}^{\prime},m_{12}^{\prime},...,m_{1s_{1}}
^{\prime},...,m_{k1}^{\prime},...,m_{ks_{k}}^{\prime}).
\]
if and only if there exists a transitive permutation representation
$\theta:\Gamma\rightarrow\Sigma_{N}$ satisfying the following conditions:

1. The permutation $\theta(x_{i})$ has precisely $s_{i}$ cycles of lengths
less than $m_{i}$, the lengths of these cycles being $m_{i}/m_{i1}^{\prime
},...,m_{i}/m_{is_{i}}^{\prime}$.

2. The Riemann-Hurwitz formula
\[
\mu(\Gamma^{\prime})/\mu(\Gamma)=N.
\]
\noindent where $\mu(\Gamma),\ \mu(\Gamma^{\prime})$ are the hyperbolic areas
of the surfaces $\mathcal{D}/\Gamma$, $\mathcal{D}/\Gamma^{\prime}$.
\end{theorem}

For $\mathcal{G}$, an abstract group isomorphic to all the Fuchsian groups of
signature $s=(h;m_{1},...,m_{k})$, the Teichm\"{u}ller space of Fuchsian
groups of signature $s$ is:

\begin{center}
$\{\rho:\mathcal{G}\rightarrow PSL(2,\mathbb{R}):s(\rho(\mathcal{G}))=s\}/$
conjugation in $PSL(2,\mathbb{R})$ \ = \ $T_{s}$.
\end{center}

The Teichm\"{u}ller space $T_{s}$ is a simply-connected complex manifold of
dimension $3g-3+k$. The modular group, $M(\Gamma)$, of $\Gamma$, acts on
$T(\Gamma)$ as $[\rho]\rightarrow\lbrack\rho\circ\alpha]$ where $\alpha\in
M(\Gamma)$. The moduli space of $\Gamma$ is the quotient space $\mathcal{M}%
(\Gamma)=T(\Gamma)/M(\Gamma)$, then $\mathcal{M}(\Gamma)$ is a complex
orbifold and its singular locus is $\mathcal{B}(\Gamma)$, called the branch
locus of $\mathcal{M}(\Gamma)$. If $\Gamma_{g}$ is a surface Fuchsian group,
we denote $\mathcal{M}_{g}=T_{g}/M_{g}$ and the branch locus by $\mathcal{B}%
_{g}$. The branch locus $\mathcal{B}_{g}$ consists of surfaces with
non-trivial symmetries for $g>2$.

If $X/Aut(X)=\mathcal{D}/\Gamma$ and genus$(X)=g$, then there is a natural
inclusion $i:T_{s}\rightarrow T_{g}:[\rho]\rightarrow\lbrack\rho^{\prime}]$, where

\begin{center}
$\rho:\mathcal{G}\rightarrow PSL(2,\mathbb{R})$, $\pi_{1}(X)\subset
\mathcal{G}$, $\rho^{\prime}=\rho\mid_{\pi_{1}(X)}:\pi_{1}(X)\rightarrow
PSL(2,\mathbb{R})$.
\end{center}

If we have $\pi_{1}(X)\vartriangleleft\mathcal{G}$, then there is a
topological action of a finite group $G=\mathcal{G}/\pi_{1}(X)$ on surfaces of
genus $g$ given by the inclusion $a:\pi_{1}(X)\rightarrow\mathcal{G}$. This
inclusion $a:\pi_{1}(X)\rightarrow\mathcal{G}$ produces $i_{a}(T_{s})\subset
T_{g}$.

The image of $i_{a}(T_{s})$ by $T_{g}\rightarrow\mathcal{M}_{g}$ is
$\overline{\mathcal{M}}^{G,a}$, where $\overline{\mathcal{M}}^{G,a}$ is the
set of Riemann surfaces with automorphisms group containing a subgroup acting
in a topologically equivalent way to the action of $G$ on $X$ given by the
inclusion $a$, see \cite{[H]}, the subset $\mathcal{M}^{G,a}\subset
\overline{\mathcal{M}}^{G,a}$ is formed by the surfaces whose full group of
automorphisms acts in the topological way given by $a$. The branch locus,
$\mathcal{B}_{g}$, of the covering $T_{g}\rightarrow\mathcal{M}_{g}$ can be
described as the union $\mathcal{B}_{g}=\bigcup_{G\neq\{1\}}\overline
{\mathcal{M}}^{G,a}$, where $\{\mathcal{M}^{G,a}\}$ is the equisymmetric
stratification of the branch locus \cite{[B]}:

\begin{theorem}
(\cite{[B]}) Let $\mathcal{M}_{g}$ be the moduli space of Riemann surfaces of
genus $g$, $G$ a finite subgroup of the corresponding modular group $M_{g} $. Then:

(1) $\overline{\mathcal{M}}^{G,a}_{g}$ is a closed, irreducible algebraic
subvariety of $\mathcal{M}_{g}$.

(2) ${\mathcal{M}}_{g}^{G,a}$, if it is non-empty, is a smooth, connected,
locally closed algebraic subvariety of $\mathcal{M}_{g}$, Zariski dense in
$\overline{\mathcal{M}}_{g}^{G,a}$.

There are finitely many strata ${\mathcal{M}}_{g}^{G,a}$.
\end{theorem}

An isolated stratum $\mathcal{M}^{G,a}$ in the above stratification is a
stratum that satisfies $\overline{\mathcal{M}}^{G,a}\cap\overline{\mathcal{M}
}^{H,b}=\varnothing$, for every group $H$ and action $b$ on surfaces of genus
$g$. Thus $\overline{\mathcal{M}}^{G,a}=\mathcal{M}^{G,a}$

Since each non-trivial group $G$ contains subgroups of prime order, we have
the following remark:

\begin{remark}
\label{primeremark}(\cite{[C]})
\[
\mathcal{B}_{g}=\underset{p\text{ prime}}{\bigcup}\overline{\mathcal{M}
}^{C_{p},a}
\]
\noindent where $\overline{\mathcal{M}}^{C_{p},a}$ is the set of Riemann
surfaces of genus $g$ with an automorphism group containing $C_{p}$, the
cyclic group of order $p$, acting on surfaces of genus $g$ in the topological
way given by $a$.
\end{remark}


\section{Isolated strata of $p$-gonal Riemann surfaces}

\begin{definition}
A Riemann surface $X$ is said to be \textit{$p$-gonal} if it admits
a $p$-sheeted covering $f:X\rightarrow\widehat{\mathbb{C}}$ onto the Riemann
sphere. If $f$ is a cyclic regular covering then $X$ is called cyclic
$p$-gonal. The covering $f$ will be called the (cyclic) $p$-gonal morphism.
\end{definition}

\noindent A cyclic $p$-gonal Riemann surface admits an equation of the form $y^p = P(x)$.

\noindent By Lemma 2.1\ in \cite{[A]}, if the surface $X_{g}$\ has genus $g\geq
(p-1)^{2}+1$, then the $p$-gonal morphism is unique.

\smallskip\noindent
We can characterize cyclic $p$-gonal Riemann surfaces using Fuchsian groups.
Let $X_{g}$ be a Riemann surface, $X_{g}$ admits a cyclic $p$-gonal morphism
$f$ if and only if there is a Fuchsian group $\Delta$ with signature
$(0;\overbrace{p,...,p}^{{\frac{2g}{p-1}}+2})$ and an index $p$ normal surface
subgroup $\Gamma$ of $\Delta$, such that $\Gamma$ uniformizes $X_{g}$ (See
\cite{[CI4]}, \cite{[CI]}). 

\smallskip\noindent
We have the following algorithm to recognize cyclic $p$-gonal surfaces: 
A surface $X_{g}$ admits a cyclic $p$-gonal morphism $f$
if and only if there is a Fuchsian group $\Delta$ with signature $
(0;m_{1},...,m_{r})$, an order $p$ automorphism $\alpha:X_{g}\rightarrow
X_{g}$, such that $\langle\alpha\rangle\leq G=Aut(X_g)$, and an epimorphism $
\theta:\Delta\rightarrow G$ with $ker(\theta)=\Gamma$ in such a way that $
\theta^{-1}(\langle\alpha\rangle)$ is a Fuchsian group with signature $(0;
\overbrace{p,...,p}^{{\frac{2g}{p-1}}+2})$. Furthermore the $p$-gonal morphism $f$ is unique
if and only if $\langle\alpha\rangle$ is normal in $G$ (see \cite{[G]}), and
Wootton \cite{[W]} has proved the following: 
\begin{lemma}\label {normal}(\cite{[W]})
With the notation above. If $G > C_p$, then $N_G(C_p) > C_p$.
\end{lemma}
 
\noindent Isolated strata $\overline{\mathcal{M}}^{C_p,a}=\mathcal{M}^{C_p,a}$ of cyclic $p$-gonal surfaces correspond to 
maximal actions of the cyclic group $C_p$. 
Isolated strata of dimension 0 where given \cite{[K]}, isolated strata of dimension 1 were studied in \cite{[CI3]}. We find here isolated strata of any dimension, consisting of $p$-gonal surfaces also:

\begin{theorem}
\label{pgonal} Let $p$ be a prime number at least seven and let $d\ge 2$. Then 
there are isolated strata of dimension $d$ consisting of $p$-gonal surfaces in $\mathcal{B}_g$ if and only if $g=(d+1)({p-1\over 2})$.
\end{theorem}

\begin{proof} First of all, an equisymmetric stratum $\mathcal{M}^{C_p,a}$ in $\mathcal{B}_p$ of dimension $d\ge 2$ of $p$-gonal Riemann surfaces is given by a monodromy 
$\theta: \Delta(0; \overbrace{p, \dots, p}^{d+3}) \to C_p$, with $\Delta$ a Fuchsian group with maximal signature (see \cite{[Si2]}). Then, a generic surface $X$ in $\mathcal{M}^{C_p,a}$ will have $C_p= Aut(X)$. The dimension of the stratum is $d={2g-p+1\over p-1}$ by the Riemann-Hurwitz formula Thus $g=(d+1)({p-1\over 2})$. 

\noindent If a surface in the stratum
has larger automorphism group $G$, then, by Lemma \ref{normal}, we can assume that $C_p$ is normal in $G$ by considering $C_p < N_G(C_p)$.

\noindent Let $X_{g}$, be a $p$-gonal surface, such that
$X_{g}\in\overline{\mathcal{M}}_{g}^{C_{p},a}$ for some action $a$, let
$\langle\alpha\rangle$ be the group of $p$-gonal automorphisms of
$X_{g}$. Consider an automorphism $b\in Aut(X)\setminus\left\langle
\alpha\right\rangle $, by Lemma \ref{normal} and \cite{[G]}, $b$ induces an
automorphism $\hat{b}$ of order $t\ge 2$ on the Riemann sphere $X_{g}/\langle
a\rangle=\widehat{\mathbb{C}}$ according to the following diagram:

\begin{center}
$
\begin{array}
[c]{ccc}
X_{g}=\mathcal{D}/\Gamma_{g} & \overset{b}{\rightarrow} & X_{g}=\mathcal{D}
/\Gamma_{g}\\
f_{a}\downarrow &  & \downarrow f_{a}\\
X_{g}/\langle\alpha\rangle=\widehat{\mathbb{C}}(P_{1},\dots,P_{k}) & \overset{\hat{b}
}{\rightarrow} & X_{g}/\langle\alpha\rangle=\widehat{\mathbb{C}}(P_{1},\dots,P_{k})
\end{array}
$
\end{center}

\noindent where $\Gamma_{g}$ is a surface Fuchsian group and $f_{a}
:X_{g}=\mathcal{D}/\Gamma_{g}\rightarrow X_{g}/\langle\alpha\rangle$ is the $p$-gonal
morphism induced by the group of automorphisms $\langle\alpha\rangle$ with
action $a$. $S=\{P_{1},\dots,P_{k}\}$ is the branch set in $\widehat{\mathbb{C}}$ of the
morphism $f_{a}$ with monodromy $\theta_{a}:\Delta(0;p,\overset{d+3}{\dots
},p)\rightarrow C_{p}$ defined by $\theta_{a}(x_{i})=\alpha^{r_{i}}$, where
$r_{i}\in\{1,\dots,p-1\}$ for $1\leq i\leq d+3$. 

Now, $\hat{b}$ induces a permutation on $S$ that either takes singular points
with monodromy $\alpha^{j}$ to points with monodromy $\alpha^{\beta(j)}$, with $\beta$ an automorphism of $C_p$, or
it acts on each subset formed by points in $S$ with same monodromy
$\alpha^{r_{j}}$. 

We construct monodromies $\theta:\Delta(0;p,\overset{d+3}{\dots
},p)\rightarrow C_{p} =\langle\alpha\rangle$, where $d=\frac{2g}{p-1}-2 \ge 2$ by the Riemann-Hurwitz
formula. We separate the monodromies in cases according to the congruence of $d$ modulus $p$.
\begin{enumerate}
\item 
$d \equiv r \not\equiv 0, 2, p-2, p-1 \,  mod(p)$ 

$\theta:\Delta(0;p,\overset{d+3}{\dots},p)\rightarrow C_{p}$ is defined by 

$\theta(x_i)=\alpha, \, 1\le i\le d; \, \theta(x_{d+1})=\alpha^{2}; \, \theta(x_{d+2})=\alpha^{p-2}; \, \theta(x_{d+3})=\alpha^{p-r}$.

\item 
$d \equiv 0 \, mod(p)$ 

$\theta:\Delta(0;p,\overset{d+3}{\dots},p)\rightarrow C_{p}$ is defined by 

$\theta(x_i)=\alpha, \, 1\le i\le d; \, \theta(x_{d+1})=\alpha^{3}; \, \theta(x_{d+2})=\alpha^{5}; \, \theta(x_{d+3})=\alpha^{p-8}$.

\item 
$d \equiv 2 \, mod(p)$ 

$\theta:\Delta(0;p,\overset{d+3}{\dots},p)\rightarrow C_{p}$ is defined by 

$\theta(x_i)=\alpha, \, 1\le i\le d; \, \theta(x_{d+1})=\alpha^{3}; \, \theta(x_{d+2})=\alpha^{p-3}; \, \theta(x_{d+3})=\alpha^{p-2}$.

\item 
$d \equiv p-2 \, mod(p)$ 

$\theta:\Delta(0;p,\overset{d+3}{\dots},p)\rightarrow C_{p}$ is defined by 

$\theta(x_i)=\alpha, \, 1\le i\le d; \, \theta(x_{d+1})=\alpha^{3}; \, \theta(x_{d+2})=\alpha^{p-3}; \, \theta(x_{d+3})=\alpha^{2}$.

\item 
$d \equiv p-1 \, mod(p)$ 

$\theta:\Delta(0;p,\overset{d+3}{\dots},p)\rightarrow C_{p}$ is defined by 

$\theta(x_i)=\alpha, \, 1\le i\le d; \, \theta(x_{d+1})=\alpha^{4}; \, \theta(x_{d+2})=\alpha^{5}; \, \theta(x_{d+3})=\alpha^{p-8}$.
\end{enumerate}
(Notice that $p-8=6$ when $p=7$ in cases 2 and 5)

\noindent We see that the given epimorphisms force $\hat{b}$ to be the
identity on $\widehat{\mathbb{C}}$. Thus, the surfaces $X_g$ do not admit 
a larger group of automorphisms than $C_p=\langle\alpha\rangle$ and the equisymmetric 
strata given by the monodromies above are isolated.
\end{proof}

Theorem \ref{pgonal} generalizes de results obtained in \cite{[BCI1]} for 
isolated strata of pentagonal Riemann surfaces,  the results in \cite{[CI3]} for one-dimensional isolated strata,  and the results in \cite{[K]} for isolated Riemann surfaces. Kulkarni \cite{[K]} showed that a branch locus $\mathcal{B}_g$ contains isolated Riemann surfaces if and only if $g=2$ or $g={p-1\over 2}$, with $p\ge 11$ a prime number. The isolated Riemann surfaces are cyclic $p$-gonal surfaces. Costa and Izquierdo \cite{[CI3]} showed that $\mathcal{B}_g$ contains one-dimensional isolated strata if and only if $g=p-1$, with $p\ge 11$ a prime number.

\begin{remark}
The isolated strata of heptagonal surfaces with dimension ${g\over 3}-1$ in $\mathcal{B}_{g}$ obtained here are 
different of the isolated strata of heptagonal surfaces and dimension ${g\over 3}-1$ obtained in \cite{[BCI2]} since the
actions determined by the monodromies are not topologically
equivalent, see \cite{[H]}.
\end{remark}

\smallskip\noindent 
In \cite {[BCI1]} we showed that $\mathcal{B}_g$ contains isolated strata of cyclic pentagonal surfaces for all even genera greater or equal eighteen. In \cite{[BI]} (see also \cite{[Bo]} and \cite{[BCIP]}) it is shown that the $\mathcal{B}_2$ contains one isolated pentagonal Riemann surface and that $\mathcal{B}_4, \, \mathcal{B}_6$ and $\mathcal{B}_8$ do not contain isolated strata of pentagonal Riemann surfaces. We study the remaining branch loci in the following proposition:

\begin{proposition}\label{five}
\end{proposition}
\begin{enumerate}
\item
 $\mathcal{B}_{10}, \, \mathcal{B}_{14}$ and $\mathcal{B}_{16}$ contain isolated strata of cyclic pentagonal Riemann surfaces.
 \item
 $\mathcal{B}_{12}$ does not contain isolated strata of cyclic pentagonal Riemann surfaces.

\end{enumerate}

\begin{proof}
\begin{enumerate}
\item
\smallskip\noindent
Consider monodromies:

$\theta_1: \Delta(0,5, \overset{7}{\dots}, 5)\to C_5=\langle\alpha\rangle$ defined by 
 $\theta_1(x_1)=\theta_1(x_2)=\theta_1(x_3)=\alpha; \, \theta_1(x_4)=\alpha^2; \, \theta_1(x_5)=\theta_1(x_6)=\alpha^3; \, \theta_1(x_7)=\alpha^4$,

 $\theta_2: \Delta(0,5, \overset{9}{\dots}, 5)\to C_5=\langle\alpha\rangle$ defined by 
 $\theta_2(x_i)=\alpha, 1\le i\le 6; \, \theta_2(x_7)=\alpha^2; \, \theta_2(x_8)=\alpha^3; \, \theta_2(x_9)=\alpha^4$,
 
$\theta_3: \Delta(0,5, \overset{10}{\dots}, 5)\to C_5=\langle\alpha\rangle$ defined by 
$\theta_3(x_1)=\alpha; \, \theta_3(x_2)=\alpha^2; \, \theta_3(x_3)=\dots=\theta_3(x_5)\alpha^3; \, \theta_3(x_6)=\dots=\theta_3(x_{10})=\alpha^4$ 

 With the same argument as in Theorem \ref{pgonal} we see that $\theta_1, \, \theta_2$ and $\theta_3$ induce isolated strata in 
$\mathcal{B}_{10}, \, \mathcal{B}_{14}$ and $\mathcal{B}_{16}$ respectively.

\item Case $\mathcal{B}_{12}$. The only possible monodromies $\theta: \Delta(0,5, \overset{8}{\dots}, 5)\to C_5=\langle\alpha\rangle$ are, up to an automorphism of $C_5$ and permuting the order of the generators of $\Delta$: 

\textbf{i)} $\theta(x_1)=\dots=\theta(x_5)=\alpha; \, \theta(x_6)=\alpha^2; \, \theta(x_7)=\theta(x_8)=\alpha^4$; 

\textbf{ii)} $\theta(x_1)=\dots=\theta(x_5)=\alpha; \, \theta(x_6)=\alpha^4; \, \theta(x_7)=\theta(x_8)=\alpha^3$; 

\textbf{iii)} $\theta(x_1)=\dots=\theta(x_4)=\alpha; \, \theta(x_6)=\dots=\theta(x_8)=\alpha^4$;

\textbf{iv)} $\theta(x_1)=\dots=\theta(x_4)=\alpha; \, \theta(x_5)=\alpha^2; \, \theta(x_6)=\theta(x_7)=\theta(x_8)=\alpha^4$;

\textbf{v)} $\theta(x_1)=\dots=\theta(x_4)=\alpha; \, \theta(x_5)=\alpha^2; \, \theta(x_6)=\theta(x_7)=\theta(x_8)=\alpha^3$;

\textbf{vi)} $\theta(x_1)=\dots=\theta(x_4)=\alpha; \, \theta(x_5)=\theta(x_6)=\alpha^2; \, \theta(x_7)=\alpha^3; \, \theta(x_8)=\alpha^4$;

\textbf{vii)} $\theta(x_1)=\theta(x_2)=\theta(x_3)=\alpha; \, \theta(x_4)=\theta(x_5)=\theta(x_6)=\alpha^2; \, \theta(x_7)=\theta(x_8)=\alpha^3$;

\textbf{viii)} $\theta(x_1)=\theta(x_2)=\theta(x_3)=\alpha; \, \theta(x_4)=\alpha^2; \,\theta(x_5)=\alpha^3; \,\theta(x_6)=\theta(x_7)=\theta(x_8)=\alpha^4$;

\textbf{ix)} $\theta(x_1)=\theta(x_2)=\alpha; \, \theta(x_3)=\theta(x_4)=\alpha^2; \,\theta(x_5)=\theta(x_6)=\alpha^3; \,\theta(x_7)=\theta(x_8)=\alpha^4$.
\end{enumerate}

\noindent With the argument in the proof of Theorem \ref{pgonal} the action of $C_5$ on the pentagonal surfaces $\mathcal{D}/Ker(\theta)$ can be extended to the action of a larger group. For instance the action of $C_5$ in case \textbf{ix)} can be extended to an action of $C_{10}, \, D_5$ or $C_5\rtimes C_4$.
\end{proof}

\begin{remark}
Theorem \ref{pgonal} and Porposition \ref{five} can be interpreted geometrically as follows: Let $(X_g^1, C_p)$ and $(X_g^2, G)$ be two Riemann surfaces  with symmetry, where $X_1$ belongs to one of the isolated strata of cyclic $p$-gonal surfaces in $\mathcal{B}_g$ and $X_g^2$ has another symmetry. Then any path of quasiconformal deformations joining $X_g^1$ and $X_g^2$ must contain surfaces without symmetry.
\end{remark}

\medskip\noindent
We consider the existence of several isolated equisymmetric strata in branch loci. 
Let $5\le p_1< p_2 <\dots < p_r$ be prime numbers. We define $\lambda = l.c.m.({p_i-1\over 2})_{i=1}^{r}$.
As a consequence of Theorem \ref{pgonal} , Theorem 3.6 in \cite{[K]} and Theorem 5 in \cite{[CI3]} we obtain:

\begin{theorem}\label{multiprime}
Let $5\le p_1< p_2 <\dots < p_r$ be prime numbers. Then, for all $g=k\, \lambda$, $k\ge 1$ and $g > 12$, the branch locus 
$\mathcal{B}_g$ contains $r$ isolated strata formed by cyclic $p_i$-gonal Riemann surfaces, $1\le i\le r$. 
\end{theorem}

\begin{proof} Observe that the conditions of Theorem \ref{pgonal} are satisfied if $g\ge {3\over2}(p_r-1)$. The conditions of Theorem 5 in \cite{[CI3]} and Theorem \ref{pgonal} are satisfied if $g = p_r-1$. Finally the conditions of Theorem \ref{pgonal}, Theorem 5 in \cite{[CI3]} and Theorem 3.6 in \cite{[K]} are satisfied if $g = {p_r-1\over 2}$. The dimension of the isolated strata of
cyclic $p_i$-gonal surfaces is $d_i= {2g\over p_i-1}-1$ by the Riemann-Hurwitz formula.

$\mathcal{B}_{12}$ does not contain isolated strata of cyclic pentagonal Riemann surfaces, it contains isolated strata of cyclic heptagonal Riemann surfaces.
\end{proof}

\smallskip\noindent
As a consequence we have: 
\begin{corollary}
Given a number $r\in \mathbb{N}$, there is an infinite number of genera $g$ such that $\mathcal{B}_g$ contains at least $r$ isolated equisymmetric strata. 
\end{corollary}
\smallskip\noindent
We finish with some examples for small genera.

\subsection{Examples}
\begin{enumerate}
\item
By Theorem 5 in \cite{[CI3]} and Proposition \ref{five}, $\mathcal{B}_{10}$ contains one isolated stratum of cyclic pentagonal surfaces of dimension four, and one 1-dimensional startum of cyclic 11-gonal surfaces.

\item
By Theorem \ref{pgonal} and Theorem 5 in \cite{[CI3]}, the smallest genus for which the branch locus contains isolated strata of cyclic heptagonal and 13-gonal Riemann surfaces is twelve.The dimensions of the isolated strata are 3 and 1 respectively. 

\item
By Theorem \ref{multiprime}, $\mathcal{B}_{20}$ contains both isolated strata of cyclic pentagonal and 11-gonal Riemann surfaces. The dimensions of the isolated strata are 9 and 3 respectively. By \cite{[K]}, $\mathcal{B}_{20}$ contains 
isolated Riemann surfaces that are cyclic 41-gonal.

\item
The smallest genus for which the branch locus contains both isolated strata of cyclic heptagonal and 11-gonal Riemann surfaces is fifteen. The dimensions of the isolated strata are 3 and 2 respectively. By \cite{[K]}, $\mathcal{B}_{15}$ contains isolated Riemann surfaces that are cyclic 31-gonal.

\item
The smallest genus $g$ for which the branch locus $\mathcal{B}_g$ contains both isolated strata of cyclic pentagonal and heptagonal Riemann surfaces is eighteen. The dimensions of the strata are 8 and 5 respectively. It contains also isolated strata of cyclic 13-gonal Riemann surfaces of dimension 3. By \cite{[CI3]} and \cite{[K]}, $\mathcal{B}_{18}$ contains one-dimensional 
isolated strata of cyclic 19-gonal surfaces and isolated cyclic 37-gonal Riemann surfaces.

\item
By Theorem \ref{multiprime}, $\mathcal{B}_{24}$ contains isolated strata of cyclic pentagonal, heptagonal,  13-gonal and 17-gonal Riemann surfaces. The dimensions of the isolated strata are 11, 7, 3 and 2 respectively. 

\item
The smallest genus for which the branch locus contains isolated strata of cyclic pentagonal, heptagonal and 11-gonal Riemann surfaces is thirty, the dimensions of these strata are 14, 9 and 5 respectively. By Theorem \ref{multiprime}, $\mathcal{B}_{30}$ contains also isolated strata of cyclic 13-gonal Riemann, surfaces with dimension 4.  By \cite{[CI3]} and \cite{[K]}, $\mathcal{B}_{30}$ contains one-dimensional isolated strata of cyclic 31-gonal surfaces and isolated cyclic 61-gonal Riemann surfaces.

\item
$\mathcal{B}_{60}$ contains isolated strata of cyclic pentagonal, heptagonal, 11-gonal, 13-gonal, 31-gonal, 41-gonal, 61-gonal surfaces, with dimensions 29, 19, 11, 9, 3, 2 and 1 respectively.

\item
$\mathcal{B}_{1000}$ contains isolated strata of cyclic pentagonal, 11-gonal, 17-gonal, 41-gonal, 101-gonal, 251-gonal and 401-gonal surfaces, with dimensions 499, 199, 124, 49, 19, 7 and 4 respectively.

\item
$\mathcal{B}_{2012}$ contains 1005-dimensional isolated strata of cyclic pentagonal surfaces.

\end{enumerate}

\end{document}